\documentclass[a4paper,12pt,reqno]{amsart}

\usepackage[T1]{fontenc}
\usepackage[utf8]{inputenc}
\usepackage{lmodern}
\usepackage[english]{babel}

\usepackage{microtype}

\usepackage[%
	left=2.5cm,       
	right=2.5cm,      
	top=3.5cm,        
	bottom=3.5cm,     
	heightrounded,    
	bindingoffset=0mm 
]{geometry}

\usepackage{functan}
\usepackage{amssymb}
\usepackage{mathtools}
\usepackage{mathrsfs}

\usepackage[hyphens]{url}
\usepackage[
colorlinks=true,
urlcolor=teal,
citecolor=magenta,
linkcolor=teal,
linktocpage,
pdfpagelabels,
bookmarksnumbered,
bookmarksopen,
breaklinks=true]
{hyperref}
\usepackage[
capitalize, 
nameinlink,
noabbrev]
{cleveref}

\usepackage{bookmark}

\usepackage[initials,msc-links,
]{amsrefs}

\usepackage[dvipsnames]{xcolor}
\usepackage{graphicx}
\usepackage{enumitem}
\usepackage{multicol}

\usepackage{tikz}
\usepackage{xfrac}

\numberwithin{equation}{section}

\theoremstyle{plain}
	\newtheorem{theorem}{Theorem}[section]
	\newtheorem{lemma}[theorem]{Lemma}
	\newtheorem{proposition}[theorem]{Proposition}

\theoremstyle{definition}

	\newtheorem{remark}[theorem]{Remark}

\newcommand{\N}{\mathbb{N}}
\newcommand{\R}{\mathbb{R}}

\renewcommand{\phi}{\varphi}
\renewcommand{\rho}{\varrho}
\renewcommand{\theta}{\vartheta}

\DeclareMathOperator{\supp}{spt}
\DeclareMathOperator{\dist}{dist}

\DeclareMathOperator{\Divv}{div}

\newcommand{\di}{\,\mathrm{d}}

\mathchardef\ordinarycolon\mathcode`\:
\mathcode`\:=\string"8000
\begingroup \catcode`\:=\active
  \gdef:{\mathrel{\mathop\ordinarycolon}}
\endgroup

\allowdisplaybreaks

\usepackage[textwidth=2.5cm]{todonotes}
\definecolor{kubin}{rgb}{0.5, 0.5, 0.8}
\definecolor{stefani}{rgb}{1.0, 0.13, 0.32}
\definecolor{saracco}{rgb}{0.7, 0.2, 0.75}

\usepackage[normalem]{ulem}
\definecolor{DarkGreen}{rgb}{0,0.5,0.1} % David

\newcommand\soutD{\bgroup\markoverwith
{\textcolor{DarkGreen}{\rule[.5ex]{2pt}{1pt}}}\ULon}
\newcommand{\Hm}[1]{\leavevmode{\marginpar{\tiny%
$\hbox to 0mm{\hspace*{-0.5mm}$\leftarrow$\hss}%
\vcenter{\vrule depth 0.1mm height 0.1mm width \the\marginparwidth}%
\hbox to
0mm{\hss$\to $\hspace*{-0.5mm}}$\\\relax\raggedright #1}}}

\newcommand{\h}{\mathscr{H}}
\newcommand{\K}{\mathscr{K}}
\newcommand{\D}{\mathscr{D}}
\newcommand{\F}{\mathscr{F}}

\usepackage[
final
]{showlabels}

\showlabels{bib}

\allowdisplaybreaks
\begin{document}

\title{On the $\Gamma$-limit of weighted fractional energies}

\author[A. Kubin]
{Andrea kubin}
\address[A. Kubin]{
Jyv\"askyl\"an Yliopisto, Matematiikan ja Tilastotieteen Laitos, Jyv\"askyl\"a, Finland
}
\email{andrea.a.kubin@jyu.fi}

\author[G.~Saracco]{Giorgio Saracco}
\address[G.~Saracco]{
Dipartimento di Matematica e Informatica, Universit\`a di Ferrara, via Machiavelli 30, 44121 Ferrara (FE), Italy
}
\email{giorgio.saracco@unife.it}

\author[G.~Stefani]{Giorgio Stefani}
\address[G.~Stefani]{
Dipartimento di Matematica, Universit\`a di Padova, via Trieste 63, 35131 Padova (PD), Italy
}
\email{giorgio.stefani@unipd.it}



\thanks{\textit{Acknowledgements}. 
The second-named and third-named authors are members of the Istituto Nazionale di Alta Matematica (INdAM), Gruppo Nazionale per l'Analisi Matematica, la Probabilit\`a e le loro Applicazioni (GNAMPA).  
The first-named author is supported by the Academy of Finland (grant agreement No.\ 314227).
The second-named author has received funding from INdAM under the INdAM--GNAMPA Project 2025 \textit{Disuguaglianze funzionali di tipo geometrico e spettrale} (grant agreement No.\ CUP\_E53\-240\-019\-500\-01).
The third-named author has received funding from INdAM under the INdAM--GNAMPA Project 2025 \textit{Metodi variazionali per problemi dipendenti da operatori frazionari isotropi e anisotropi} (grant agreement No.\ CUP\_E53\-240\-019\-500\-01) and from the European Union -- NextGenerationEU and the University of Padua under the 2023 STARS@UNIPD  Starting Grant Project \textit{New Directions in Fractional Calculus -- NewFrac} (grant agreement No.\ CUP\_C95\-F21\-009\-990\-001).
}
	
\begin{abstract}
Given $p\in[1,\infty)$ and a bounded open set $\Omega\subset\R^d$ with Lipschitz boundary, we study the $\Gamma$-convergence of the weighted fractional seminorm 
\begin{equation*}
[u]_{s,p,f}^p
=
\int_{\R^d} \int_{\R^d} \frac{|\tilde u(x)- \tilde u(y)|^p}{\|x-y\|^{d+sp}}\,f(x)\,f(y)\di x\di y,
\end{equation*}
as $s\to1^-$ for $u\in L^p(\Omega)$, where $\tilde u=u$ on $\Omega$ and $\tilde u=0$ on $\R^d\setminus\Omega$.
Assuming that $(f_s)_{s\in(0,1)}\subset L^\infty(\R^d;[0,\infty))$ and $f\in\mathrm{Lip}_b(\R^d;(0,\infty))$ are such that $f_s\to f$ in $L^\infty(\R^d)$ as $s\to1^-$, we show that $(1-s)[u]_{s,p,f_s}^p$ $\Gamma$-converges to the Dirichlet $p$-energy weighted by $f^2$.
In the case $p=2$, we also prove the convergence of the corresponding gradient flows. 
\end{abstract}
	
%%%%%%%%%%%%%%%%%%%%%%%%%
%%%%%%%%%%%%%%%%%%%%%%%%%
%%%%%%%%%%%%%%%%%%%%%%%%%
%%%%%%%%%%%%%%%%%%%%%%%%%

 \hspace{-2cm}
 {
 \begin{minipage}[t]{0.7\linewidth}
 \begin{scriptsize}
 \vspace{-2cm}
This is a pre-print of an article published in \emph{Proc.\ Roy.\ Soc.\ Edinburgh Sect.\ A}. The final authenticated version is available online at: \href{https://doi.org/10.1017/prm.2025.10100}{https://doi.org/10.1017/prm.2025.10100}
 \end{scriptsize}
\end{minipage} 
}

\maketitle

%%%%%%%%%%%%%%%%%%%%%%%%%
%%%%%%%%%%%%%%%%%%%%%%%%%
%%%%%%%%%%%%%%%%%%%%%%%%%
%%%%%%%%%%%%%%%%%%%%%%%%%

\section{Introduction}

\subsection{Framework}

We let $d \in \N$, $s \in (0,1)$, and $p\in[1,\infty)$. 
Given a nonnegative \emph{weight} $f \in L^\infty(\R^d;[0,\infty))$, our aim is to study the $\Gamma$-convergence as $s\to1^-$ of the non-ho\-mo\-ge\-neous (or \emph{weighted} by $f$) $s$-fractional $p$-seminorm
\begin{equation}
\label{eq:funct-intro}
[u]_{s,p,f}^p 
= 
\int_{\R^d} \int_{\R^d} \frac{| u(x)- u(y)|^p}{\|x-y\|^{d+sp}}\, f(x)\,f(y)\di x\di y,
\end{equation}
for $u\in L^p(\R^d)$.

The convergence as $s\to1^-$ of~\eqref{eq:funct-intro} in the case $f\equiv 1$---for which we use the shorthand $[\,\cdot\,]_{s,p}^p=[\,\cdot\,]_{s,p,1}^p$---has been deeply studied in recent years, both in the pointwise and in the $\Gamma$-sense.
Since the literature is very vast, here we limit ourselves to a non-comprehensive list of results which are closer to the spirit of the present work.

The pointwise limit of the seminorm $[\,\cdot\,]^p_{s,p}$ as $s\to1^-$ is a notable instance of the celebrated \emph{Bourgain--Brezis--Mironescu} (BBM, for short) \emph{formula}~\cites{BBM01,Dav02}, yielding that $(1-s)[\,\cdot\,]_{s,p}$ converges to the Dirichlet $p$-energy up to a multiplicative constant.
After the seminal contributions~\cites{BBM01,Dav02}, the BBM formula has been extensively studied in several directions, see~\cites{Pon04a,Pon04b,DDG24} for more general results and~\cites{LS11,LS14} for extensions to arbitrary domains.
We also refer to~\cites{Fan25,Lud14-S,Lud14-p} for anisotropic fractional energies and to~\cites{DDP24,GS25} for \textit{sharp} conditions for the validity of the BBM formula. 

The $\Gamma$-convergence of $(1-s)[\,\cdot\,]_{s,p}^p$ to the Dirichlet $p$-energy as $s\to1^-$ has been established in~\cite{BPS16} for every $p\in(1,\infty)$, in~\cites{CdLKNP23} only for $p=2$, and in~\cite{GS25} for every $p\in[1,\infty)$.
We also refer to~\cites{BBM01,Pon04a} for similar results on bounded open sets.

The geometric case $p=1$ deserves special mention, due to the link with the (relative) fractional perimeter, see~\cites{AdPM01,P20,BP19,KlM24,GS25,Lom19} for closely related results in this direction.
We also refer to~\cites{CN20,dLKP22,KPT2024} for higher-order convergence results.

Beyond the case $f\equiv1$, the asymptotic behavior of~\eqref{eq:funct-intro} and of similarly-defined energies has been studied for some particular weights, see~\cites{CCLP23,DL21} for the \emph{Gaussian} framework and~\cite{K24} for weights depending on negative powers of the distance from the boundary.

The aim of the present paper is to investigate the asymptotic behavior of the weighted seminorm~\eqref{eq:funct-intro} as $s\to1^-$ as the weight $f$ is also allowed to vary with respect to the parameter~$s$.
This is motivated by the recent interest in the extension of BBM-type formulas beyond the isotropic setting in order to address possible applications to non-isotropic frameworks~\cite{DDP24}.
Besides, our $\Gamma$-convergence result can be interpreted as a suitable extension to the weighted setting of the ones obtained in~\cites{CdLKNP23,BPS16}.

Our main result, \cref{thmGammaconvs1stab} below, deals with the $\Gamma$-convergence of the energy~\eqref{eq:funct-intro} with respect to a uniformly converging family of weights $(f_n)_{n\in\N}$ in $L^\infty(\R^d;[0,\infty))$, whose limit $f$ is in $\mathrm{Lip}_b(\R^d;(0,\infty))$. 
Precisely, we prove that the $\Gamma$-limit is given by 
\begin{equation*}
u \mapsto 
\begin{cases}
\displaystyle
K_{d,p}
\|\nabla u\|_{p,f^2}^p
=
K_{d,p}
\int_{\Omega} f^2\,\| \nabla u\|^p \di x, 
&
\text{for}\ p \in (1,\infty),
\\[3ex]
\displaystyle
K_{d,1}
\|Du\|_{1,f^2}
=
K_{d,1}
\int_{\Omega} f^2\di|Du|,
&
\text{for}\ p=1,
\end{cases}
\end{equation*}
where, 
for every $p\in[1,\infty)$ (and here $\Gamma$ being Euler's \emph{Gamma function}), 
\begin{equation}
\label{eq:def_Kdp}
K_{d,p}
=
\frac{1}{p}\int_{\partial B_1} \vert x \cdot \mathrm{e}_d \vert^p\, \mathrm{d} \mathcal{H}^{d-1}(x)
=
\frac{2\pi^{\frac{d-1}2}}{p}
\frac{\Gamma\left(\frac{p+1}2\right)}{\Gamma\left(\frac{N+p}2\right)},
\end{equation}
see~\cite{BSZ26}*{Lem.~2.1}.
Here and below, given a measurable function $u\colon\Omega\to\R$ on an open set $\Omega\subset\R^d$, we define $\tilde u\colon\R^d\to\R$ be such that $\tilde u=u$ on $\Omega$ and $\tilde u=0$ on $\R^d\setminus\Omega$. 

\begin{theorem}[$\Gamma$-convergence with weights]
\label{thmGammaconvs1stab}
Let $p\in[1,\infty)$, $(f_n)_{n\in\N}\subset L^\infty(\R^d;[0,\infty))$ and $f\in\mathrm{Lip}_b(\R^d;(0,\infty))$ be such that $f_n\to f$ in $L^\infty(\R^d)$, $\Omega\subset \mathbb{R}^d$ be a bounded open set with Lipschitz boundary and $(s_n)_{n\in\N}\subset (0,1)$ be such that $s_n \to 1^-$.
\begin{enumerate}[label=(\roman*)]
\item
\label{item:compactnessstab}
\emph{(Compactness)} 
If $(u^n)_{n \in \N}\subset L^p(\Omega)$ is such that
\begin{equation}
\label{eq:equibddstab}
\sup_{n \in \N}\left( (1-s_n)[\tilde u^n]^p_{s_n,p,f_n}+ \| u^n \|_{L^p(\Omega)}^p\right)<\infty,
\end{equation} 
then, up to a subsequence, $u^n\to u$ in $L^p(\Omega)$ for some $u\in W^{1,p}_0(\Omega)$ if $p>1$ or $u \in BV(\Omega)$ if $p=1$.
\item
\label{item:G-liminfstab}
\emph{($\Gamma$-liminf inequality)}
If $(u^n)_{n\in\N}\subset L^p(\Omega)$ is such that $u^n\to u$ in $L^p(\Omega)$ for some $u\in W^{1,p}_0(\Omega)$ if $p>1$, or $u\in BV(\Omega)$ if $p=1$, then
\begin{equation}
\label{linfsto1stab}
\begin{split}
&K_{d,p} \| \nabla u  \|^p_{p,f^2}
\le
\liminf_{n\to \infty}(1-s_n) [\tilde u^n]^p_{s_n,p,f_n} \quad \text{for $p>1$},\\
& K_{d,1} \| D u  \|_{1,f^2}
\le
\liminf_{n\to \infty}(1-s_n) [\tilde u^n]_{s_n, 1, f_n} \quad \text{for $p=1$}.
\end{split}
\end{equation}
\item
\label{item:G-limsupstab}
\emph{($\Gamma$-limsup inequality)} 
If $u\in W^{1,p}_0(\Omega)$ if $p>1$, or $u\in BV(\Omega)$ if $p=1$, then there exists $(u^n)_{n\in\N}\subset L^p(\Omega)$ such that $u^n\to u$ in $L^p(\Omega)$ and
\begin{equation}
\label{lsupsto1stab}
\begin{split}
&K_{d,p} \| \nabla u  \|^p_{p,f^2}
=
\lim_{n\to \infty} (1-s_n)[\tilde u^n]^p_{s_n,p,f_n} \quad \text{for $p>1$},\\
&K_{d,1} \| D u  \|_{1,f^2}
=
\lim_{n\to \infty} (1-s_n)[\tilde u^n]_{s_n, 1, f_n} \quad \text{for $p =1$}.
\end{split}
\end{equation}
\end{enumerate}
\end{theorem}

\subsection{Convergence of flows}

In the case $p=2$, the $\Gamma$-convergence result obtained in \cref{thmGammaconvs1stab} can be complemented with a stability result for the corresponding parabolic flows associated to the energies, see \cref{11genstab} below.  
Here and below, given a weight $f\in L^\infty(\R^d;[0,\infty))$, we define the weighted Laplacian of $u \in H^1_0(\Omega)$ as 
\begin{equation}
\label{def:D1f}
(-\mathfrak{D})^{f}u 
= 
-2\Divv(f^2\nabla u)
\end{equation}
in the distributional sense in duality with $C^\infty_c(\Omega)$ functions.
Moreover, given $u\in L^2(\Omega)$ such that $[u]_{s,2,f}<\infty$, we define the weighted fractional $s$-Laplacian of $u$ as
\begin{equation}
\label{def:Dsf}
(-\mathfrak{D})^{s,f}u (x)
=
4f(x)\,\mathrm{p.v.}\int_{\R^d}\frac{\tilde u(x)-\tilde u(y)}{\|x-y\|^{d+2s}}\,f(y)\,\mathrm{d} y
\end{equation}  
again in the distributional sense in duality with $C^\infty_c(\Omega)$ functions.
Note that, in the unweighted case $f\equiv1$, up to a multiplicative constant, the operators~\eqref{def:D1f} and~\eqref{def:Dsf} become the usual Laplacian and fractional $s$-Laplacian operators, respectively.

\begin{theorem}[Stability of parabolic flows]
\label{11genstab}
Let $(f_n)_{n \in \N}$, $f$, $\Omega$, and $(s_n)_{n\in\N}$ be as in \cref{thmGammaconvs1stab}.  
If $(u^n_0)_{n\in\N}\subset L^2(\Omega)$ is such that $u^n_0\to u^\infty_0$ in $L^2(\Omega)$ for some function $u^\infty_0\in L^2(\Omega)$, $[\tilde{u}^n_{0}]_{s_n,2,f_n}<\infty$ for every $n\in\N$, and 
\begin{equation*}
\sup_{n\in\N}
(1-s_n)[\tilde{u}^n_0]_{s_n,2,f_n}^2
<\infty,
\end{equation*}
then the following hold:
\begin{enumerate}[label=(\roman*),itemsep=1ex]

\item 
$u^\infty_0\in H^1_0(\Omega)$;

\item
for every $T>0$ and for every $n\in\N$, the problem
\[
\begin{dcases}
 \dot u(t) =(1-s_n)(-\mathfrak{D})^{s_n,f_n}u(t),  \quad\textrm{for a.e.\ }t\in(0,T), \\[1ex]
u(0)=u^n_0,
\end{dcases}
\]
admits a unique solution $u_n\in H^1([0,T];  L^2(\Omega))$ such that
\begin{equation*}
(-\mathfrak{D})^{s_n,f_n}u_n(t)\in L^2(\Omega)
\quad
\text{for a.e.}\ t\in(0,T);
\end{equation*} 
\item
the problem
\[
\begin{dcases}
\dot u(t) =K_{d,2}(-\mathfrak{D})^{f} u(t),  \quad\text{for a.e.}\ t\in[0,\infty),\\[1ex]
u(0)=u^\infty_0,
\end{dcases}
\]
admits a unique solution $u_\infty\in H^1([0,T];H^1_0(\Omega))$;

\item
$(u_n)_{n\in\N}$ weakly converges to $u_\infty$ in $H^1([0,T];L^2(\Omega))$.
\end{enumerate}
Moreover, if
\begin{equation*}
\lim_{n\to \infty} (1-s_n)[\tilde{u}^n_0]^2_{s_n,2,f_n}
= 
K_{d,2}\| \nabla u^\infty_0\|^2_{2,f^2},
\end{equation*} 
then $(u_n)_{n\in\N}$ strongly converges to $u_\infty$ in $H^1([0,T];L^2(\Omega))$ and also
\[
u_n(t) \xrightarrow{L^2} u_\infty(t)
\quad 
\text{and}
\quad 
(1-s_n)[\tilde{u}_n(t)]_{s_n,2,f_n}\to K_{d,2}\|\nabla u_\infty(t)\|_{2, f^2}
\quad
\text{for every}\ t\in[0,T].
\]
\end{theorem}

\subsection{Organization of the paper}

The paper is organized as follows. The notation and some useful preliminary results are detailed in \cref{sec:prelim}. The proof of \cref{thmGammaconvs1stab} is given in \cref{sec:stab}, while that of \cref{11genstab} can be found in \cref{sec:flow}.

\section{Preliminaries}
\label{sec:prelim}

\subsection{Notation}
We briefly detail the main notation used throughout the paper.

The symbol $C(*,\cdots,*)$ indicates a  generic positive constant that depends on $*,\cdots,*$ only and
may change from line to line.

We let $d\in\N$ and work in the $d$-dimensional Euclidean space $\R^d$.
We let $x\cdot y$ be the Euclidean inner product between $x,y\in\R^d$ and $\|x\|$ be the Euclidean norm of $x$.

We let $B_r(x)$ be the open ball in $\R^d$	of center $x\in\R^d$ and radius $r>0$, and we use the shorthand $B_r= B_r(0)$.
Given an open set $A\subset\R^d$, we let $A^c=\R^d\setminus A$ be the complement of $A$, $\partial A$ be the topological boundary of $A$ and, for every $t >0$,
\begin{equation}
\label{ingrass}
A_{t}= \{ x \in \mathbb{R}^d: \; \dist(x;A)<t\}.
\end{equation} 
Throughout the paper, we let $\Omega\subset \R^d$ be a bounded open set with Lipschitz boundary.

We let $\mathcal{L}^d$ be the $d$-dimensional Lebesgue measure and $\mathcal{H}^\alpha$ be the $\alpha$-dimensional Hausdorff measure for every  $\alpha\in[0,d]$.
We set $\omega_d=\mathcal{L}^d(B_1)$, so that $\mathcal H^{d-1}(\partial B_1)=d\,\omega_d$.
Throughout the paper, all functions and sets are tacitly assumed to be  $\mathcal L^d$-measurable.

Let $p\in[1,\infty)$ and $f \in L^\infty(\R^d;[0,\infty))$.
Given $m\in\N$ and $v\colon\Omega \to \R^m$, we let
\begin{equation}
\label{eq:Lpf}
\|v\|_{p,f}
= 
\left(   \int_{\Omega} \| v(x)\|^p \,f(x)\,\mathrm{d}x \right)^{\frac{1}{p}}\in[0,\infty],
\end{equation}
and we use the shorthand $\|v\|_p=\|v\|_{p,1}$.
We thus let
\begin{equation*}
[L^p_f(\Omega)]^m = \left\{v\colon \Omega \to \R^m : \|v\|_{p,f}<\infty \right\}.
\end{equation*}
When $m=1$, we simply write $L^p_f(\Omega)$.
We point out that if additionally $f$ takes values in $(0,\infty)$, under our standing assumptions on $f$ and $\Omega$, the spaces $L^p(\Omega)$ and $L^p_f(\Omega)$ are equivalent, with
\begin{equation}\label{eq:Lp-embedding}
(\operatorname{ess\,inf}_\Omega f)\,\|v\|_{p}^p 
\le 
\|v\|^p_{p,f} 
\le 
\|f\|_\infty \|v\|^p_{p}.
\end{equation}

Given $u\colon\Omega\to\R$, we define $\tilde u\colon\R^d\to\R$  by letting $\tilde u=u$ in $\Omega$ and $\tilde u=0$ in $\R^d\setminus\Omega$. 
Thus, given $s\in(0,1)$, we define 
\begin{equation}
\label{eq:spf_seminorm}
[u]_{s,p,f} 
= 
\left( \int_{\mathbb{R}^d} \int_{\mathbb{R}^d} \frac{|\tilde u(x)-\tilde u(y)|^p}{\|x-y\|^{d+sp}}\, f(x)\,f(y)\,\mathrm{d}x\,\mathrm{d}y \right)^{\frac{1}{p}}
\end{equation}
for every $u\colon\Omega\to\R$ and we use the shorthand $[u]_{s,p}=[u]_{s,p,1}$.

Finally, we let $W^{1,p}_0(\Omega)$ for $p>1$, be the closure of $C^\infty_c(\Omega)$ functions with respect to the Sobolev $p$-norm
$u\mapsto\|u\|_p^p+\int_{\R^d}\|\nabla u\|_{p}^p\di x$, while $BV(\Omega)$ the weak$^*$ closure of $C^\infty_c(\Omega)$ with respect to the Sobolev $1$-norm. 
We also set
\begin{equation}
\label{eq:semiBVg}
\|D u\|_{1,f}
= 
\int_{\R^d} f\di|Du|,
\end{equation}
whenever $u \in BV(\Omega)$ and $f\in L^\infty(\R^d;(0,\infty))$.

\subsection{Compactness and characterization}

We recall the following well-known compactness result, see~\cite{Bre11book}*{Thm.~4.26} for example.
Here and below, we let $
\tau_h w (\,\cdot\,)= w(\,\cdot\,+h)$ for every  $h\in \mathbb{R}^d$ and $w\in L^p(\mathbb{R}^d)$.

\begin{theorem}
\label{frechetkolmthm}
Let $p\in[1,\infty)$.
If $(v^n)_{n \in \N}\subset L^p(\Omega)$ is such that 
\begin{equation*}
\sup_{n\in\N}\|v^n\|_p<\infty
\quad
\text{and}
\quad
\lim_{h \to  0} 
\sup_{n \in \N} 
\| \tau_{h}\tilde v^n-\tilde v^n \|_{L^p(\mathbb{R}^d)}
=
0,
\end{equation*}
then, up to a subsequence, $v^n\to v$ in $L^p(\Omega)$ for some $v\in L^p(\Omega)$. 
\end{theorem}

We also recall the following well-known characterization of Sobolev and $BV$ functions, see~\cite{Bre11book}*{Prop.~9.3 and Rem.~6} for example.

\begin{theorem}
\label{carsobvshift}
Let $p\in[1,\infty)$ and $v \in L^p(\mathbb{R}^d)$.
The following are equivalent:
\begin{enumerate}[label=(\roman*),itemsep=1ex]

\item
$v \in W^{1,p}(\mathbb{R}^d)$ for $p>1$ or $v \in BV(\R^d)$ for $p=1$;

\item
$\sup_{\|h\|\le1}
\| \tau_{h}v-v \|_{p}<\infty$.
\end{enumerate} 
\end{theorem}

\section{Proof of \texorpdfstring{\cref{thmGammaconvs1stab}}{Theorem 1.1}}
\label{sec:stab}

Throughout this section, we let $p\in [1,\infty)$, $(s_n)_{n\in\N}\subset(0,1)$ be such that $s_n\to1^-$, and $(f_n)_{n\in\N}\subset L^\infty(\R^d;[0,\infty))$ and $  f\in\mathrm{Lip}_b(\R^d; (0,\infty))$ be such that $f_n\to f$ in $L^\infty(\R^d)$.

We preliminarily prove \cref{thmGammaconvs1stab} in the case $f_n=f$ for every $n\in\N$.
We restate our result in this particular case for better clarity.

\begin{theorem}
\label{thmGammaconvs1}
The following hold.
\begin{enumerate}[label=(\roman*),itemsep=1ex]

\item
\label{item:compactness}
\emph{(Compactness)} 
If $(u^n)_{n \in \N}\subset L^p(\Omega)$ is such that
\begin{equation}
\label{eq:equibdd}
\sup_{n \in \N} 
\left(
(1-s_n)[\tilde u^n]^p_{s_n,p,f}+ \| u^n \|_{L^p(\Omega)}^p
\right)
<\infty,
\end{equation} 
then, up to a subsequence, $u^n\to u$ in $L^p(\Omega)$ for some $u\in W^{1,p}_0(\Omega)$ if $p \in (1,\infty)$ or $ u \in BV(\Omega)$ if $p=1$.

\item
\label{item:G-liminf}
\emph{($\Gamma$-liminf inequality)} 
If $(u_n)_{n \in \N}\subset L^p(\Omega)$ is such that $u_n\to u$ in $L^p(\Omega)$ for some $u\in L^p(\Omega)$, then
\begin{equation}
\label{linfsto1}
\begin{split}
&K_{d,p} \| \nabla u  \|^p_{p,f^2}
\le
\liminf_{n\to \infty}(1-s_n) [\tilde u^n]^p_{s_n,p,f} \quad \text{for $p \in (1,\infty)$},\\
& K_{d,1} \| D u  \|_{1,f^2}
\le
\liminf_{n\to \infty}(1-s_n) [\tilde u^n]_{s_n, 1, f} \quad \text{for $p=1$}.
\end{split}
\end{equation}

\item
\label{item:G-limsup}
\emph{($\Gamma$-limsup inequality)} 
If $u\in W^{1,p}_0(\Omega)$ if $p>1$, or $u\in BV(\Omega)$ if $p=1$, then there exists $(u^n)_{n \in \N}\subset L^p(\Omega)$ such that $u_n\to u$ in $L^p(\Omega)$ and
\begin{equation}
\label{lsupsto1}
\begin{split}
&K_{d,p} \| \nabla u  \|^p_{p,f^2}
=
\lim_{n\to \infty} (1-s_n)[\tilde u^n]^p_{s_n,p,f} \quad \text{for $p \in (1,\infty)$},\\
&K_{d,1} \| D u  \|_{1,f^2}
=
\lim_{n\to \infty} (1-s_n)[\tilde u^n]_{s_n,1, f} \quad \text{for $p =1$.}
\end{split}
\end{equation}
\end{enumerate}
\end{theorem}

The proof of the three statements~\ref{item:compactness}, \ref{item:G-liminf}, and~\ref{item:G-limsup} of \cref{thmGammaconvs1} is split across  \cref{subsec:compactness,subsec:liminf,subsec:limsup}.
The proof of \cref{thmGammaconvs1stab} is given in \cref{subsec:stab}. 

\subsection{Proof of  \texorpdfstring{\cref{thmGammaconvs1}\ref{item:compactness}}{Theorem 3.1(i)}}
\label{subsec:compactness}

We adapt the strategy of~\cite{AdPM01} to our setting.
To this aim, we need two preliminary results.
The first one is the following, which generalizes~\cite{AdPM01}*{Prop.~5} to any $p\in[1,\infty)$ and  weighted $L^p$ norms.
We also refer to~\cite{CdLKNP23}*{Prop.~2.4} for the case $p=2$ without weights. 

\begin{proposition}
Let $f \in \mathrm{Lip}_b(\R^d; (0,\infty))$.
There exists $C=C(d,p)>0$ such that
\begin{equation}
\label{prop1formenunc}
\| \tau_{h}v-v \|_{L^p_f(E)}^{p} 
\leq 
C\frac{\|h\|^{p}}{\rho^{d+p}}\int_{B_\rho}\|\tau_{y}v-v \|_{L^p_f(E_{\|h\|})}^p \,\mathrm{d} y
\end{equation}
for every $v \in L^p(\mathbb{R}^d)$, $h \in \mathbb{R}^d$, $\rho\in(0,\|h\|]$, and every bounded open set $E \subset\mathbb{R}^d$, where $E_{\|h\|}$ is defined according to the notation in~\eqref{ingrass}.
\end{proposition}

\begin{proof}
The proof closely follows the one of~\cite{AdPM01}*{Prop.~5}. 
Let $\varphi \in C_{c}^1(B_1)$ be  such that\begin{equation}
\label{eq:hp_varphi_mollifier}
\varphi\geq 0
\quad 
\text{and}
\quad 
\int_{B_1} \varphi(x)\,\mathrm{d} x=1.
\end{equation}
For every $\rho>0$, we let $U_\rho$ and $V_\rho$ be defined as
\begin{align*}
U_\rho(x)
&
=
\frac{1}{\rho^d} \int_{B_\rho} v(x+y) \varphi\left(\frac{y}{\rho}\right)\,\mathrm{d} y,
\\
V_\rho(x)
&
=\frac{1}{ \rho^d}\int_{B_\rho}(v(x)-v(x+y))\varphi\left(\frac{y}{\rho}\right)\,\mathrm{d} y,
\end{align*}
for every $x\in\R^d$.	
Owing to~\eqref{eq:hp_varphi_mollifier}, we have that $
v(x)
= 
U_\rho(x)+V_\rho(x)$  for every $\rho>0$ and $x\in\R^d$, so that 
\begin{equation}
\label{aa1}
|\tau_hv(x)-v(x)|^p
\le
3^p\big(
|U_\rho(x+h) - U_\rho(x)|^p
+
|V_\rho(x)|^p
+
|V_\rho(x+h)|^p
\big).
\end{equation}
We now estimate each term in the right-hand side of~\eqref{aa1} separately. 
Concerning the second and third term, by Jensen's inequality,
we can estimate
\begin{equation}
\label{prop1form2dim}
\vert V_\rho(\xi) \vert^{p} 
\leq 
\frac{\omega_d}{\rho^d}\,\|\varphi \|^p_{\infty} \int_{B_\rho} \vert v(\xi)-\tau_{y}v(\xi) \vert^p\, \mathrm{d} y,
\end{equation}
for every $\xi\in\R^d$.
Instead, concerning the first term, by the change of variables $z=x+y$, we can rewrite
\[
U_\rho(x)
=
\frac{1}{\rho^d}\int_{B_\rho(x)}v(z)\,\varphi\left(\frac{z-x}{\rho}\right) \,\mathrm{d} z.
\]
Thus, owing to the fact that $\phi((z-\cdot)\rho^{-1})\in C^1_c(B_{\rho}(x))$, we can integrate by parts and get  
\begin{equation*}
\begin{split}
\nabla U_\rho(x)
=
&
-\frac{1}{\rho^{d+1}}\int_{B_\rho(x)} v(z)\,\nabla\varphi\left(\frac{z-x}{\rho}\right) \,\mathrm{d} z 
\\
=
&
-\frac{1}{\rho^{d+1}}\int_{B_\rho(x)}(v(z)-v(x))\,\nabla\varphi\left(\frac{z-x}{\rho}\right) \,\mathrm{d} z
\\
=
&
-\frac{1}{\rho^{d+1}}\int_{B_\rho}(v(x+y)-v(x))\,\nabla\varphi\left(\frac{y}{\rho}\right) \,\mathrm{d} y.
\end{split}
\end{equation*}
Therefore, by the Fundamental Theorem of Calculus and by Jensen's inequality, we obtain 
\begin{align}
\label{prop1form3dim}
\vert U_\rho(x&+h) - U_\rho(x)\vert^{p} 
\leq 
\| h \|^{p} \int_{0}^{1} \vert \nabla U_\rho(x+th) \vert^p \,\mathrm{d} t 
\nonumber
\\
&
\leq \omega_d^{p-1}\frac{\|h\|^p}{\rho^{d+p}}\| \nabla\varphi \|_{\infty}^p\int_{0}^{1} \int_{B_\rho} \vert \tau_yv(x+th)-v(x+th)\vert^p \,\mathrm{d}y\,\mathrm{d}t.	
\end{align}
Now, using that $\rho<\|h\|$ and combining~\eqref{aa1}, \eqref{prop1form2dim}, and~\eqref{prop1form3dim}, we get that
\begin{align}
\vert \tau_h v(x)-v(x)\vert^p 
&
\leq
C \frac{\|h\|^p}{\rho^{d+p}} \int_{0}^{1} \int_{B_\rho} \vert \tau_y v(x+th)-v(x+th)\vert^p \,\mathrm{d}y\,\mathrm{d}t
\nonumber
\\
&
+ C \frac{\|h\|^p}{\rho^{d+p}} \int_{B_\rho}  \vert \tau_y v(x)-v(x)\vert^p \,\mathrm{d}y
\nonumber
\\
&
+ C \frac{\|h\|^p}{\rho^{d+p}}\int_{B_\rho} \vert \tau_y v(x+h)-v(x+h)\vert^{p} \,\mathrm{d}y,
\label{aa2}
\end{align}
where we have set
\[
C=C(d,p)=(3 \max\{\|\varphi \|_{\infty} ; \|\nabla\varphi \|_{\infty}\})^p\, \omega_d.
\]
Multiplying inequality~\eqref{aa2} by $f(x)$ and integrating with respect to $x\in E$, the claim immediately follows by Fubini's Theorem.
We omit the simple details.
\end{proof}

We can now pass to the following result, which extends~\cite{AdPM01}*{Prop.~4} to any $p\in[1,\infty)$ also in the case of weighted $L^p$ norms.

\begin{proposition}
\label{prop2enunciato}
Let $f\in \mathrm{Lip}_b(\R^d; (0,\infty))$. 
There exists $C=C(d,p)>0$ such that 
\begin{equation*}
\| \tau_{h}v-v \|_{L^p_f(E)}^{p} 
\leq 
C (1-s) \|h\|^{sp} \int_{B_{\|h\|}} \frac{\|\tau_{y}v-v \|_{L^{p}_f(E_{\| h \|})}^{p}}{\|y\|^{d+sp}} \,\mathrm{d}y
\end{equation*}
for every $v \in L^p(\mathbb{R}^d)$, $h \in \mathbb{R}^d$, and every bounded open set $E \subset \mathbb{R}^d$,
where $E_{\|h\|}$ is defined according to the notation  in~\eqref{ingrass}.
\end{proposition}

The proof of \cref{prop2enunciato} requires the following  Hardy-type inequality, which is taken from~\cite{AdPM01}*{Prop.~6}.

\begin{lemma}\label{Hardytipodis}
If $\phi\colon\mathbb{R}\to [0,\infty)$ is a Borel function, then 
\begin{equation}
\label{pastasciutta}
\int_{0}^{r} \frac{1}{\rho^{d+l+1}} \int_{0}^{\rho} \phi(t)\,\mathrm{d}t\,\mathrm{d}\rho 
\leq 
\frac{1}{d+l}\int_{0}^{r} \frac{\phi(t)}{t^{d+l}}\,\mathrm{d}t,
\end{equation}
for every $l,r\ge 0$.
\end{lemma} 

Actually, in the proof of \cref{prop2enunciato} we use the weaker estimate
\begin{equation}
\label{eq:hardy}
\int_{0}^{r} \frac{1}{\rho^{d+l+1}} \int_{0}^{\rho} \phi(t)\,\mathrm{d}t\,\mathrm{d}\rho 
\leq 
\frac{1}{d}\int_{0}^{r} \frac{\phi(t)}{t^{d+l}}\,\mathrm{d}t,
\end{equation}
that is, we can ignore the dependence on $l$ in the prefactor in the right-hand side of~\eqref{pastasciutta}.

\begin{proof}[Proof of \cref{prop2enunciato}]
We let $\phi_v\colon[0,\|h\|]\to\R$ be defined as 
\begin{equation}
\label{eq:def_phi_v}
\phi_v(t)
= 
\int_{\partial B_t} \|\tau_{y} v - v  \|_{L^{p}_f (E_{\|h\| })}^{p} \,\mathrm{d}\mathcal{H}^{d-1}(y),
\end{equation}
for all $t>0$.
Owing to~\eqref{prop1formenunc} and to the definition in~\eqref{eq:def_phi_v}, we can estimate
\begin{equation}
\label{prop2fomr1dim}
\| \tau_{h}v-v \|_{L_f^{p}(E)}^{p} 
\leq 
C\frac{\|h\|^{p}}{\rho^{d+p}} \int_{0}^{\rho} \phi_v(t)\,\mathrm{d}t,
\end{equation}
for some $C=C(d,p)>0$. 
We now multiply both sides of~\eqref{prop2fomr1dim} by $\rho^{-1+p-sp}$ and integrate in the interval $ [0,\|h\|]$ with respect to $\rho$, getting
\[
\| \tau_{h}v-v \|_{L_f^{p}(E)}^{p} 
\leq 
C p(1-s) \frac{\|h\|^{p}}{\|h\|^{p-sp}} \int_0^{\|h\|}\frac{1}{\rho^{d+sp+1}} \int_{0}^{\rho} \phi_v(t)\,\mathrm{d}t \,\mathrm{d}\rho.
\]
By exploiting~\eqref{eq:hardy} with $l=sp$ and $\phi=\phi_v$, we thus obtain that 
\begin{equation*}
\| \tau_{h}v-v \|_{L^{p}_f(E)}^{p} 
\leq 
C(1-s) \|h\|^{sp} \int_{0}^{\|h\|}\frac{\phi_v(t)}{t^{d+sp}}\,\mathrm{d}t,
\end{equation*}
and the conclusion follows from the very definition of $\phi_v$.
\end{proof}

We are now ready to detail the proof of the compactness statement~\ref{item:compactness} in \cref{thmGammaconvs1}.

\begin{proof}[Proof of \cref{thmGammaconvs1}\ref{item:compactness}]
Given $h\in\R^d$ such that $\|h\| < 1$, we have
$\Omega_{\|h\|}\Subset \Omega_1 \Subset (\Omega_1)_{\|h\|}$
(recall the notation in~\eqref{ingrass}).
We can hence set
\[
c = c(\Omega, f) =\inf_{(\Omega_1)_{\|h\|}} f>0,
\]
and observe that
\[
c^2 \| \tau_{h}\tilde{u}^n-\tilde{u}^n \|^p_{L^{p}(\R^d)}
=
c^2 \| \tau_{h}\tilde{u}^n-\tilde{u}^n \|^p_{L^{p}(\Omega_{\|h\|})} 
\le
c \| \tau_{h}\tilde{u}^n-\tilde{u}^n \|^p_{L^{p}_{f}(\Omega_1)}.
\]
By \cref{prop2enunciato} applied on $\Omega_1$ and by the previous inequality, we have
\[
c^2 \| \tau_{h}\tilde{u}^n-\tilde{u}^n \|^p_{L^{p}(\R^d)}
\le
C(1-s_n)\|h\|^{s_n p} \int_{B_{\|h\|}} \frac{ c \|\tau_{y}\tilde{u}^n-\tilde{u}^n \|_{L^{p}_f((\Omega_1)_{\| h \|})}^{p}}{\|y\|^{d+s_n p}} \,\mathrm{d}y,
\]
where $C=C(d,p)>0$.
Now, explicitly writing down the $L^p_f$ norm on the right-hand side, swapping order of integration, performing the change of variables $y=\xi-x$, and bounding $c$ with $f(\xi)$, we obtain that
\begin{align*}
c^2 \| \tau_{h}\tilde{u}^n&-
\tilde{u}^n \|^p_{L^{p}(\R^d)}
\le
C(1-s_n) \|h\|^{s_np} \int_{B_{\|h\|}} \int_{(\Omega_1)_{\|h\|}} \frac{|\tilde{u}^n(x+y)-\tilde{u}^n(x)|^p}{\|y\|^{d+s_n p}} c f(x) \,\mathrm{d}x \,\mathrm{d}y
\\
&
\leq
C(1-s_n) \|h\|^{s_np} \int_{(\Omega_1)_{\|h\|}} \int_{B_{\|h\|}(x)}  \frac{|\tilde{u}^n(\xi)-\tilde{u}^n(x)|^p}{\|\xi-x\|^{d+s_n p}} f(\xi)f(x)\,\mathrm{d}\xi \,\mathrm{d}x 
\\
&
\le
C(1-s_n) \|h\|^{s_np} [\tilde u^n]^p_{s_n, p,f}.
\end{align*}
Dividing by $c^2$ and owing to our equiboundedness assumption~\eqref{eq:equibdd}, we get that
\begin{equation} 
\label{teocompform1}
\| \tau_{h}\tilde{u}^n-\tilde{u}^n \|_{L^{p}(\R^d)}
\leq
C(d,p,M, \Omega, f)\|h\|^{s_n},
\end{equation}
for all $h\in\R^d$ such that $\|h\|\le 1$. Thus, owing to~\eqref{eq:equibdd} and to~\eqref{teocompform1}, we can apply \cref{frechetkolmthm} and find $u \in L^{p}(\mathbb{R}^d)$ such that, up to a subsequence, $\tilde u^n \to  u$ in $L^p(\R^d)$. 
Furthermore, since $\tilde{u}^n = 0$ for all $n\in\N$ on $\R^d\setminus \Omega$, we also have that $u=0$ on $\mathbb{R}^d \setminus \Omega$.
Finally, letting $n\to\infty$ in~\eqref{teocompform1}, we have
\[
\| \tau_{h}u-u \|_{L^{p}(\R^d)}
\leq C(d,p,M,\Omega,f)\|h\|,
\]
for all $h$ with $\|h\|\le 1$, so that $u\in W^{1,p}(\R^d)$ for $p>1$ or $u\in BV(\R^d)$ for $p=1$ by \cref{carsobvshift}. 
Since $u=0$ on $\R^d\setminus \Omega$ and $\Omega$ has Lipschitz boundary, we get that $u|_\Omega\in W^{1,p}_0(\Omega)$ for $p>1$, or $u|_\Omega\in BV(\Omega)$ for $p=1$, concluding the proof.
\end{proof}

\subsection{Proof of  \texorpdfstring{\cref{thmGammaconvs1}\ref{item:G-liminf}}{Theorem 3.1(ii)}}
\label{subsec:liminf}

We adapt the strategy of the proof of~\cite{CdLKNP23}*{Thm.~2.1} to our setting. 
To this aim, we need some preliminaries.

Let us begin with some notation. 
We let $Q=(-1,1)^d$.
Consequently, given $\gamma>0$,  for every $i \in \gamma \mathbb{Z}^d$ and $a\ge0$, we let $Q_i^a= i+ aQ$. 
Note that, if $a=\gamma$, then the family of cubes $(Q^\gamma_i)_{i}$ is a tiling of $\R^d$.
Moreover, since $\Omega$ is bounded, the set
\begin{equation*}
I_\gamma
=
\{i\in\gamma\mathbb Z^d : \mathcal L^d(Q^\gamma_i\cap \Omega)>0\}
\end{equation*} 
is finite.
In addition, given $f\colon\R^d\to\R$, we let $f_i^a= \inf_{Q_i^a}f$. Notice that, whenever $f\in\mathrm{Lip}(\R^d;(0,\infty))$, then $f_i^a>0$.
Finally, we let $\eta\in C_{c}^\infty(B_1)$ be such that $\eta\ge 0$ and $\int_{B_1}\eta\,\mathrm{d}x=1$ and, for every $\varepsilon>0$, we set $\eta_\varepsilon(\,\cdot\,)=\varepsilon^{-d}\eta(\,\cdot\,/\varepsilon)$.
Accordingly, we let $u_\varepsilon= u \ast \eta_{\varepsilon}$ for every $\varepsilon>0$ and $u\in L^1_{\rm loc}(\R^d)$.

We can now prove the following preliminary estimate.
\begin{lemma}
\label{marimba}
Let $p\in [1, \infty)$, $f \in \mathrm{Lip}_b(\R^d;(0,\infty))$. There exist $\varepsilon,\beta,\gamma>0$ with $\varepsilon \ll \beta \ll \gamma$ such that
\begin{equation}
\begin{split}
\int_{Q_i^{(1-\beta) \gamma}} \int_{Q_i^{(1-\beta) \gamma}} 
&
\frac{|\tilde{u}_\varepsilon(x)- \tilde{u}_{\varepsilon}(y)|^p }{\|x-y\|^{d+sp}} (f_i^{\gamma})^2 \,\mathrm{d}x \,\mathrm{d}y 
\\
&
\quad\qquad
\leq 
\int_{Q_i^{\gamma}} \int_{Q_i^{\gamma}} \frac{|\tilde u(x)- \tilde u(y)|^p}{\|x-y\|^{d+sp}}f(x)f(y) \,\mathrm{d}x \,\mathrm{d}y
\end{split}
\end{equation}
holds for all $i\in \gamma\mathbb{Z}^d$ and $ u \in L^p(\Omega)$.
\end{lemma}

\begin{proof}
Clearly, we can choose $\varepsilon$, $\beta$, and $\gamma$ such that
\begin{equation}
\label{eq:cube_trans_dil_inclusion}
Q_i^{(1-\beta)\gamma}+z \subset Q_i^{\gamma}
\end{equation}
for all $ z \in B_\varepsilon$ and all $i\in \gamma\mathbb{Z}^d$. Indeed, $$Q_i^{(1-\beta)\gamma}+z \subset \{ y \in \R^d : \dist(y,Q_i^{(1-\beta)\gamma}  ) < \varepsilon \} \subset Q_{i}^{(1-\beta) \gamma+ \varepsilon}$$ and therefore, to get \eqref{eq:cube_trans_dil_inclusion}, it is enough to choose $\varepsilon<\beta\gamma$ to ensure $(1-\beta)\gamma+\varepsilon<\gamma$.
Using the definition of convolution, Jensen's inequality, changing order of integration, performing the change of variables $x-z=\xi$ and $y-z=\zeta$, owing to~\eqref{eq:cube_trans_dil_inclusion}, changing once again order of integration, using the fact that $\eta_\varepsilon$ has unit $L^1$ norm, and finally owing to the definition of $f_i^{\gamma}=\inf_{Q_i^\gamma}f$, we obtain
\begin{align*}
\int_{Q_i^{(1-\beta) \gamma}} &
\int_{Q_i^{(1-\beta) \gamma}} \frac{|\tilde{u}_\varepsilon(x)- \tilde{u}_{\varepsilon}(y)|^p }{\|x-y\|^{d+sp}} (f_i^{\gamma})^2 \,\mathrm{d}x \,\mathrm{d}y 
\\
&
\leq 
\int_{Q_i^{(1-\beta) \gamma}} \int_{Q_i^{(1-\beta) \gamma}} \int_{B_\varepsilon} \frac{|\tilde{u}(x-z)- \tilde{u}(y-z)|^p }{\|x-z-(y-z)\|^{d+sp}} \eta_{\varepsilon}(z) (f_i^{\gamma})^2 \,\mathrm{d}z \,\mathrm{d}x \,\mathrm{d}y 
\\
&
\leq 
\int_{B_\varepsilon} \int_{Q_i^{(1-\beta) \gamma}+z} \int_{Q_i^{(1-\beta) \gamma}+z}  \frac{|\tilde{u}(\xi)- \tilde{u}(\zeta)|^p }{\|\xi-\zeta\|^{d+sp}} \eta_{\varepsilon}(z) (f_i^{\gamma})^2  \,\mathrm{d}\xi \,\mathrm{d}\zeta \,\mathrm{d}z
\\
&
\leq
\int_{Q_i^{\gamma}} \int_{Q_i^{\gamma}}  \frac{|\tilde{u}(\xi)- \tilde{u}(\zeta)|^p }{\|\xi-\zeta\|^{d+sp}} \,f(\xi)\,f(\zeta) \,\mathrm{d}\xi \,\mathrm{d}\zeta,
\end{align*}
concluding the proof.
\end{proof}

We are now ready to prove the $\liminf$ statement~\ref{item:G-liminf} in \cref{thmGammaconvs1}.

\begin{proof}[Proof of \texorpdfstring{\cref{thmGammaconvs1}\ref{item:G-liminf}}{Theorem~\ref{thmGammaconvs1}(ii)}]
Let $(u^n)_{n \in \N}\subset L^p(\Omega)$ be such that $u^n\to u$ in $L^p(\Omega)$ for some $u\in L^p(\Omega)$. 
As a consequence, $\sup_{n\in\N}\|u^n\|_{p}<\infty$. 
Furthermore, we can also assume that
\[
\liminf_{n\to \infty}(1-s_n) [\tilde u^n]^p_{s_n,p,f}<\infty,
\]
otherwise inequality~\eqref{linfsto1} is trivially true. Therefore, we can assume the validity of~\eqref{eq:equibdd}, which, in turn, implies that $u\in W^{1,p}_0(\Omega)$ for $p>1$, or $u\in BV(\Omega)$ for $p=1$. 
By \cref{marimba}, there exist $\varepsilon,\beta,\gamma>0$ with $\varepsilon \ll \beta \ll \gamma$ such that
\begin{equation}
\label{111220241}
\begin{split}
\int_{Q_i^{\gamma}} \int_{Q_i^{\gamma}}
&
\frac{|\tilde u^n(x)- \tilde u^n(y)|^p}{\|x-y\|^{d+s_np}} f(x)f(y) \,\mathrm{d}y \,\mathrm{d}x
\\
&\quad
\geq
\int_{Q_i^{(1-\beta) \gamma}} \int_{Q_i^{(1-\beta) \gamma}}
\frac{|(\tilde{u}^n)_\varepsilon(x)- (\tilde{u}^n)_{\varepsilon}(y)|^p}{\|x-y\|^{d+s_np}} (f_i^{\gamma})^2 \,\mathrm{d}y \,\mathrm{d}x,
\end{split}
\end{equation}
for every $ i \in \gamma\mathbb{Z}^d$ and $n \in \N$. 
We now perform a first-order Taylor expansion of 
$(\tilde{u}^n)_\varepsilon$.
Precisely, owing to the uniform bound on the $p$-norms granted by~\eqref{eq:equibdd} and the boundedness of $\Omega$, we can estimate
\begin{align}
|\nabla (\tilde{u}^n)_\varepsilon(x)\cdot(x-y)| 
&
\le
|(\tilde{u}^n)_\varepsilon(x)-(\tilde{u}^n)_\varepsilon(y)| + \frac 12 |\langle D^2(\tilde{u}^n)_\varepsilon (\xi) (x-y), (x-y) \rangle |
\nonumber
\\
&
\le
|(\tilde{u}^n)_\varepsilon(x)-(\tilde{u}^n)_\varepsilon(y)| +C(\Omega)\|u^n\|_{1}\,\|\eta_\varepsilon\|_{C^{2}(\R^d)}\,\|x-y\|^{2}
\nonumber
\\
&
\le
|(\tilde{u}^n)_\varepsilon(x)-(\tilde{u}^n)_\varepsilon(y)| +C(\varepsilon,M,\Omega)\|x-y\|^{2},
\label{eq:ornitottero}
\end{align}
where $\xi$ belongs to the segment from $x$ to $y$. 
Now, assuming $\|x-y\|$ small enough (which is always possible by taking $\gamma$ small enough), since $\tilde{u}^n_\varepsilon$ is locally Lipschitz, we have
\begin{align}
\label{bernard1} 
|(\tilde{u}^n)_\varepsilon(x)-(\tilde{u}^n)_\varepsilon(y)|
&+
C(\varepsilon,M,\Omega)\|x-y\|^{2}
\nonumber
\\
&
\le
\|\nabla(\tilde{u}^n)_\varepsilon\|_\infty \| x-y \| + C(\varepsilon,M,\Omega) \| x-y \|^2
\nonumber
\\
&
=
C(\Omega)\|u^n\|_1 \|\eta_\varepsilon\|_{C^1(\R^d)}\| x-y \| + C(\varepsilon,M,\Omega) \| x-y \|^2
\nonumber
\\
&
=
C(\varepsilon,M,\Omega) \| x-y \|.
\end{align}
Taking $p$-th powers in~\eqref{eq:ornitottero}, using that $(t_0+t)^p= t_0^p+p(t_0+\tau)^{p-1}t $ with $ \tau \in (0,t)$, and finally owing to~\eqref{bernard1}, we get
\[
\left|\nabla (\tilde{u}^n)_\varepsilon(x)\cdot(x-y)\right|^p \leq |(\tilde{u}^n)_\varepsilon(x)-(\tilde{u}^n)_\varepsilon(y)|^p +C(p,\varepsilon,M,\Omega)\|x-y\|^{p+1}.
\]
Therefore, plugging the above inequality in the inner integral in the right-hand side of~\eqref{111220241}, we get that
\begin{equation}
\label{eq:split_2_pieces}
\int_{Q_i^{(1-\beta)\gamma}} 
\frac{|(\tilde{u}^n)_\varepsilon(x)-(\tilde{u}^n)_\varepsilon(y)|^p}{\|x-y\|^{d+s_np}} (f_i^{\gamma})^2 \,\mathrm{d}y
\ge
I' + I'',
\end{equation}
having set
\begin{equation}
\label{eq:def_2_pieces}
\begin{split}
I'
&
=
\int_{Q_i^{(1-\beta)\gamma}}\frac{|\nabla (\tilde{u}^n)_\varepsilon(x)\cdot(x-y)|^p}{\|x-y\|^{d+s_np}}(f_i^{\gamma})^2 \,\mathrm{d}y,
\\
I''
&
=
- C \int_{Q_i^{(1-\beta)\gamma}}
\frac{\|x-y\|^{p+1}}{\|x-y\|^{d+s_np}} (f_i^{\gamma})^2 \,\mathrm{d}y.
\end{split}
\end{equation}
We now estimate the two terms in~\eqref{eq:def_2_pieces} separately.
On the one hand, we have
\begin{align*}
I''
&
\ge
- C(p,\varepsilon, M, \Omega) \|f\|^2_\infty \int_{B_{\delta_1}} \frac{\mathrm{d}\xi}{\|\xi\|^{d+s_np-p-1}}
\\
&
=
-C(p,\varepsilon, M, \Omega, f) d\omega_d \int_0^{\delta_1} \frac{\mathrm{d}\rho}{\rho^{s_np-p}}
=
-\frac{C(d,p,\varepsilon, M, \Omega, f)}{1+p-s_np} \delta_1^{1+p-s_np}
\end{align*}
for $x\in Q_i^{(1-\beta)\gamma}$, where we have set $\delta_1 = 2\sqrt{d}(1-\beta)\gamma$.
Integrating the above inequality over the cube $Q_i^{(1-\beta)\gamma}$ with respect to $x$, we obtain
\[
\int_{Q_i^{(1-\beta)\gamma}} I'' \,\mathrm{d}x
\ge
- \frac{C(d,p,\varepsilon, M, \Omega, f)}{1+p-s_np} \big|Q_i^{(1-\beta)\gamma}\big| \delta_1^{1+p-s_np},
\]
so that\begin{equation}
\label{eq:estimate-int-I''}
\liminf_{n \to  \infty} (1-s_n)\int_{Q_i^{(1-\beta)\gamma}} I'' \,\mathrm{d}x
\ge
0.
\end{equation}

On the other hand, calling $\delta_2=\delta_2(x)=\dist\big(x; \partial Q_i^{(1-\beta)\gamma}\big)$, taking the normalization $\nu(x)= \nabla (\tilde{u}^n)_\varepsilon(x)/\|\nabla (\tilde{u}^n)_\varepsilon(x)\|$, applying the change of variables $z=x-y$, exploiting the invariance of the integral on $B_{\delta_2}$ with respect to rotations, applying the change of variable $z\mapsto z\delta_2^{-1}$, and observing that, for every $\nu\in\mathbb S^{d-1}$ 
\begin{equation}
\label{eq:def2_Kdp}
(1-s_n)\int_{B_1} \frac{|z\cdot \nu|^p}{\|z\|^{d+s_np}} \,\mathrm{d}z = (1-s_n) \int_{0}^1 \rho^{p-s_n p -1}\, \int_{\partial B_1} \vert \eta \cdot \nu \vert^p \,\mathrm{d} \eta \, \mathrm{d} \rho = K_{d,p},
\end{equation}
we get that
\begin{align*}
I' 
&
=
|\nabla (\tilde{u}^n)_\varepsilon(x)|^p (f_i^{\gamma})^2 \int_{Q_i^{(1-\beta)\gamma}}\frac{|\nu(x)\cdot(x-y)|^p}{\|x-y\|^{d+s_np}} \,\mathrm{d}y
\\
&
\geq
|\nabla (\tilde{u}^n)_\varepsilon(x)|^p (f_i^{\gamma})^2 \int_{B_{\delta_2}} \frac{|\nu(x)\cdot z|^p}{\|z\|^{d+s_np}} \,\mathrm{d}z
= 
\frac{K_{d,p} |\nabla (\tilde{u}^n)_\varepsilon(x)|^p \delta_2^{p(1-s_n)}(f_i^{\gamma})^2 }{1-s_n}.
\end{align*}
Multiplying the above inequality by $(1-s_n)$, integrating over the cube $Q_i^{(1-\beta)\gamma}$ with respect to $x$, taking the $\liminf$ as $n\to\infty$, owing to Fatou's lemma, we get that
\[
\liminf_{n \to \infty}(1-s_n)
\int_{Q_i^{(1-\beta)\gamma}} I' \,\mathrm{d}x
\ge
K_{d,p} \int_{Q_i^{(1-\beta)\gamma}} \liminf_{n \to   \infty}\big( |\nabla (\tilde u^n)_\varepsilon(x)|^p \delta_2^{p(1-s_n)} \big) (f_i^{\gamma})^2\,\mathrm{d}x.
\]
Since $(u^n)_\varepsilon\to u_\varepsilon$ in any Sobolev norm as $n \to \infty$, up to passing to a suitable subsequence, $\nabla(\tilde u^n)_\varepsilon(x)\to\nabla\tilde u_\varepsilon(x)$ for a.e.\ $x\in\R^d$. 
Therefore, we get that
\begin{equation}
\label{eq:estimate-int-I'}
\liminf_{n \to   \infty}(1-s_n)
\int_{Q_i^{(1-\beta)\gamma}} I' \,\mathrm{d}x
\ge
K_{d,p} \int_{Q_i^{(1-\beta)\gamma}} |\nabla \tilde u_\varepsilon(x)|^p  (f_i^{\gamma})^2\,\mathrm{d}x.
\end{equation}
Thence, by taking the $\liminf$ as $n\to\infty$ of $(1-s_n)$ times ~\eqref{111220241} and owing to~\eqref{eq:split_2_pieces}--\eqref{eq:estimate-int-I'}, we have
\begin{equation*}
\begin{split}
\liminf_{n\to\infty}  (1-s_n)\int_{Q_i^{\gamma}} \int_{Q_i^{\gamma}} \frac{|\tilde u^n(x)- \tilde u^n(y)|^p}{\|x-y\|^{d+sp}}
&
f(x)f(y) \,\mathrm{d}y \,\mathrm{d}x
\\
&
\ge
K_{d,p} \int_{Q_i^{(1-\beta)\gamma}} |\nabla \tilde u_\varepsilon(x)|^p  (f_i^{\gamma})^2\,\mathrm{d}x
\end{split}
\end{equation*}
whenever $\varepsilon\ll\beta\ll\gamma\ll1$. 
Now first letting $\varepsilon \to 0^+$, and then $\beta\to 0^+$, we get
\begin{equation}
\label{eq:ineq_on_cubes}
\begin{split}
 \liminf_{n\to\infty}  (1-s_n)\int_{Q_i^{\gamma}} \int_{Q_i^{\gamma}} &  \frac{|\tilde u^n(x)- \tilde u^n(y)|^p}{\|x-y\|^{d+sp}}
f(x)f(y) \,\mathrm{d}y \,\mathrm{d}x
\\
&
\ge
K_{d,p} \int_{Q_i^\gamma} |\nabla \tilde u(x)|^p  (f_i^{\gamma})^2\,\mathrm{d}x, \quad \text{ if $ p \in (1,\infty)$},\\
\liminf_{n\to \infty}  (1-s_n)\int_{Q_i^{\gamma}} \int_{Q_i^{\gamma}} & \frac{|\tilde u^n(x)- \tilde u^n(y)|}{\|x-y\|^{d+s}}
f(x)f(y) \,\mathrm{d}y \,\mathrm{d}x
\\
&
\ge
K_{d,1} \int_{Q_i^\gamma}   (f_i^{\gamma})^2\,\mathrm{d} \vert D u \vert(x), \quad \text{ if $ p =1$}.
\end{split}
\end{equation}

We notice now that, by the Lebesgue's Dominated Convergence Theorem,
\begin{equation*}
    \begin{split}
  \lim_{\gamma \to 0^+} \sum_{i \in I_\gamma}	K_{d,p} \int_{Q_i^\gamma} |\nabla\tilde{u}(x)|^p  (f_i^{\gamma})^2 \,\mathrm{d}x 
= 
K_{d,p} \| \nabla u  \|^p_{p,f^2}    
\\
\text{and}\quad
 \lim_{\gamma \to 0^+} \sum_{i \in I_\gamma}	K_{d,1} \int_{Q_i^\gamma}   (f_i^{\gamma})^2 \,\mathrm{d} \vert D u \vert  
= 
K_{d,1} \| D u  \|_{1,f^2}. 
    \end{split}
\end{equation*}
Pairing this with~\eqref{eq:ineq_on_cubes} allows to conclude. Indeed, if $p \in (1,\infty)$ (the case $p =1$ being analogous and thus omitted), then we have 
\begin{align*}
K_{d,p} 
&
\int_{\Omega} \vert \nabla u(x) \vert^p  f^2(x)\,\mathrm{d}x
\\
&
\le
\lim_{\gamma \to0^+}
\sum_{i \in I_\gamma}
\liminf_{n \to  \infty}  (1-s_n)\int_{Q_i^{\gamma}} \int_{Q_i^{\gamma}} \frac{|\tilde u^n(x)- \tilde u^n(y)|^p}{\|x-y\|^{d+s_np}}
f(x)f(y) \,\mathrm{d}y \,\mathrm{d}x
\\
&
=
\lim_{\gamma \to  0^+}\liminf_{n \to  \infty} \sum_{i \in I_\gamma}(1-s_n)\int_{Q_i^{\gamma}} \int_{Q_i^{\gamma}} \frac{|\tilde u^n(x)- \tilde u^n(y)|^p}{\|x-y\|^{d+s_np}}
f(x)f(y) \,\mathrm{d}y \,\mathrm{d}x
\\
&
\le
\lim_{\gamma \to  0^+}\liminf_{n \to  \infty} (1-s_n)\int_{\Omega} \int_{\Omega} \frac{|\tilde u^n(x)- \tilde u^n(y)|^p}{\|x-y\|^{d+s_np}}
f(x)f(y) \,\mathrm{d}y \,\mathrm{d}x
\\
&
=
\liminf_{n \to  \infty}(1-s_n) \int_{\R^d} \int_{\R^d} \frac{|\tilde u^n(x)- \tilde u^n(y)|^p}{\|x-y\|^{d+s_np}}
f(x)f(y) \,\mathrm{d}y \,\mathrm{d}x,
\end{align*}
which is exactly the claimed inequality.
\end{proof}

\subsection{Proof of \texorpdfstring{\cref{thmGammaconvs1}\ref{item:G-limsup}}{Theorem 3.1(iii)}}
\label{subsec:limsup}

We begin with the following preliminary result, establishing the $\limsup$ inequality~\ref{item:G-limsup} in \cref{thmGammaconvs1} for smooth functions supported in~$\Omega$.

\begin{theorem}
\label{teocitaziones1limsup}
If $v \in C_c^{\infty}(\mathbb{R}^d)$, then
\begin{equation}
\label{convpunt}
\begin{split}
\lim_{s \to  1^-} (1-s) \int_{\mathbb{R}^d} \int_{\mathbb{R}^d} \frac{|v(x)-v(y)|^p}{\|x-y\|^{d+sp}}
&
f(x)f(y)\,\mathrm{d}x\,\mathrm{d}y
\\
&
=
K_{d,p}\int_{\R^d}\|\nabla v(z)\|^p f^2(z) \,\mathrm{d}z.
\end{split}
\end{equation}
\end{theorem}

\begin{proof}
Let us write
\begin{equation}
\label{asdasd1}
\int_{\R^d} \int_{\R^d} \frac{|v(x)-v(y)|^p}{\|x-y\|^{d+sp}} f(x) f(y) \,\mathrm{d}x \,\mathrm{d}y
=
I_1+I_2+I_3,
\end{equation}
where 
\begin{equation}
\label{asdasd2}
\begin{split}
I_1
&
= 
\int_{\R^d} \int_{\{ \|x-y\| < 1 \}} \frac{|v(x)-v(y)|^p}{\|x-y\|^{d+sp}} f^2(y) \,\mathrm{d}x \,\mathrm{d}y,
\\
I_2
&
=
\int_{\R^d} \int_{\{ \|x-y\| < 1 \}} \frac{|v(x)-v(y)|^p}{\|x-y\|^{d+sp}} f(y)(f(x)-f(y)) \,\mathrm{d}x \,\mathrm{d}y,
\\
I_3
&
= 
\int_{\R^d} \int_{\{ \|x-y\| \geq 1 \}} \frac{|v(x)-v(y)|^p}{\|x-y\|^{d+sp}} f(x) f(y) \,\mathrm{d}x \,\mathrm{d}y.
\end{split}
\end{equation}
We estimate the three terms separately.
Concerning $I_3$, we have that
\begin{align}
I_3
& 
\le
2^p\|f\|^2_\infty \int_{\R^d} \int_{\{ \|x-y\| \geq 1\}} \frac{|v(x)|^p+|v(y)|^p}{\|x-y\|^{d+sp}} \,\mathrm{d}x \,\mathrm{d}y 
\nonumber
\\
& 
\leq 
2^{p+1}\| f \|^2_{\infty}\int_{\R^d} \int_{B_1(y)^c} \frac{|v(y)|^p}{\|x-y\|^{d+sp}} \,\mathrm{d}x \,\mathrm{d}y 
\nonumber
\\
&
=
2^{p+1}\| f \|^2_{\infty} \| v \|^p_{p} \int_{1}^{\infty} \frac{d \omega_d}{\rho^{1+sp}} \,\mathrm{d} \rho 
=
\frac{2^{p+1}}{sp} \| f \|^2_{\infty} \|v\|^p_{p}. 
\label{asdasd3}
\end{align}
For $I_2$ instead, we can write, recalling the notation in~\eqref{ingrass},
\begin{align}
I_2 
&
\leq 
\| f \|_{\infty} \| \nabla f \|_{\infty}\int_{\R^d} \int_{B_1(y)} \frac{|v(x)-v(y)|^p}{\|x-y\|^{d+sp-1}} \,\mathrm{d}x \,\mathrm{d}y
\nonumber
\\
&
\leq  
C(f) \| \nabla v \|_{\infty}^p  \int_{(\supp{v})_1} \int_{B_1(y)} \frac{1}{\|x-y\|^{d-1-p+sp}} \,\mathrm{d}x \,\mathrm{d}y
\nonumber
\\
& 
\leq  
C(f,v) d\omega_d |(\supp{v})_1| \int_{0}^{1} \frac{1}{\rho^{-p+sp}} \,\mathrm{d}\rho
\nonumber
\\
&
\leq
\frac{C(d,f,v)}{p-sp+1}
\le
C(d,p,f,v).
\label{asdasd6}
\end{align}
We are thus left with estimating $I_1$. 
To this aim, let us observe that
\begin{equation*}
\vert v(x)-v(y) \vert^p 
\le
\vert \nabla v(y) \cdot(x-y) \vert^p+  \| D^2 v \|_{\infty}^{2p} \, \|x-y\|^{2p}.
\end{equation*}
Thus $I_1\le I'_1+I''_1$, where, owing to the fact that $v$ has compact support and recalling the notation in~\eqref{ingrass},
\begin{equation*}
\begin{split}
I'_1
&
= 
\int_{\R^d} \int_{B_1(y)} \frac{\vert \nabla v(y) \cdot (x-y)\vert^p}{\|x-y\|^{d+sp}} f^2(y) \,\mathrm{d}x \,\mathrm{d}y,
\\
I''_1
&
= 
\int_{(\supp{v})_1} \int_{B_1(y)} \frac{\| D^2 v \|_{\infty}^{2p}}{\|x-y\|^{d+sp-2p}} f^2(y) \,\mathrm{d}x \,\mathrm{d}y.
\end{split}
\end{equation*}
Now, on the one hand, we have that
\begin{align}
I''_1 
&
\leq
\|f\|_{\infty}^2 \|D^2 v\|_{\infty}^{2p} \int_{(\supp{v})_1}\int_{B_1(y)} \frac{1}{\|x-y\|^{d+sp-2p}}\,\mathrm{d}x \,\mathrm{d}y
\nonumber
\\
&
\leq
C(f,v) d\omega_d |(\supp{v})_1| \int_{0}^{1}\rho^{p-1+p(1-s)} \,\mathrm{d}\rho
\nonumber\\
&
=
\frac{C(d,f,v)}{2p-sp} \le C(d,p,f,v).
\label{asdasd9}
\end{align}
By the non-negativity of $I''_1,I_2,I_3$, and by~\eqref{asdasd3}--\eqref{asdasd9}, we thus get that 
\begin{equation}
\label{eq:vanishing_pt}
\lim_{s\to 1^-}(1-s)(I''_1+I_2+I_3) = 0.
\end{equation}
Therefore, in order to conclude, we are left with showing that $\lim_{s\to1^-} (1-s)I'_1$ equals the right hand side of~\eqref{convpunt}. 
Indeed, by the change of variables $z=x-y$, by setting $\nu(y)=\nabla v(y)/\|\nabla v(y)\|$, and using \eqref{eq:def2_Kdp}, we have
\begin{align}
{I'}_1 
&
=
\int_{\R^d} \int_{B_1} \frac{|\nabla v(y) \cdot z|^p}{\|z\|^{d+sp}} f^2(y) \,\mathrm{d}z \,\mathrm{d}y
\nonumber
\\
&
=
\int_{\R^d} \|\nabla v(y)\|^p f^2(y)\int_{B_1} \frac{|\nu(y)\cdot z|^p}{\|z\|^{d+sp}} \,\mathrm{d}z \,\mathrm{d}y
\nonumber
\\
&
= 
\frac{K_{d,p}}{(1-s)}
\int_{\R^d}\|\nabla v(x)\|^p f^2(y)\, \mathrm{d}y.
\label{asdasd10}
\end{align}
The conclusion hence follows by combining~\eqref{eq:vanishing_pt} and~\eqref{asdasd10}.
\end{proof}

\begin{remark}
We underline that, in the chain of inequalities leading to~\eqref{asdasd6}, the Lipschitz regularity of the weight $f$ is crucial in order to ensure the finiteness of the integral in $\rho$. 
If $f$ were only $\alpha$-H\"older, for some $\alpha\in(0,1)$, one would end up with $\rho$ to the power $-(\alpha-1-p+sp)$, which is not integrable in a neighborhood of the origin.
\end{remark}

We are now ready to prove the $\limsup$ statement~\ref{item:G-limsup} in \cref{thmGammaconvs1}.

\begin{proof}[Proof of \texorpdfstring{\cref{thmGammaconvs1}\ref{item:G-limsup}}{Theorem~\ref{thmGammaconvs1}(iii)}]
Let $u\in W^{1,p}_0(\Omega)$ for $p>1$, or $ u \in BV(\Omega)$ for $p=1$. 
By the density of $C_c^{\infty}(\Omega)$ in  $W^{1,p}_0(\Omega)$ for $p>1$, or in $BV(\Omega)$ for $p=1$, we can find $(v^k)_{k\in\N} \subset C_c^{\infty}(\Omega)$ such that $ v^{k} \to  u$ in $W^{1,p}_0(\Omega)$ for $p>1$, or in energy in $BV(\Omega)$ for $p=1$. 
In view of \cref{teocitaziones1limsup}, we thus get that 
\begin{equation*}
\begin{aligned} 
\lim_{n \to  \infty} (1-s_n) [\tilde v^k]^p_{s_n,p,f}
&
=
\lim_{n\to\infty} (1-s_n) \int_{\mathbb{R}^d} \int_{\mathbb{R}^d} \frac{|\tilde v^k(x)-\tilde v^k(y)|^p}{\|x-y\|^{d+s_np}}f(x) f(y)\,\mathrm{d}x \,\mathrm{d}y
\\
&
=  
K_{d,p} \int_{\R^d} \|\nabla \tilde v^k(x)\|^2 f^2(x)\,\mathrm{d}x
\\
&
= 
K_{d,p} \int_{\Omega} \|\nabla v^k(x)\|^p f^2(x) \,\mathrm{d}x
\end{aligned}
\end{equation*}
for every $k\in\N$.
The conclusion hence follows by a standard diagonal argument.
\end{proof}

\subsection{Proof of \texorpdfstring{\cref{thmGammaconvs1stab}}{Theorem 1.1}}
\label{subsec:stab}

We can now prove \cref{thmGammaconvs1stab} in full generality.
Indeed, the result easily follows by combining \cref{thmGammaconvs1} with
\cref{Lem108012025,Lem208012025} below (under the same standing assumptions as stated at the very beginning of \cref{sec:stab}).

\begin{lemma}
\label{Lem108012025}
If $(u^n)_{n\in\N}\subset L^p(\Omega)$ is such that 
\begin{equation}
\label{eq:bound_lemmetto_seq_pesi}
\sup_{n \in \N}\left( (1-s_n)[\tilde u^n]^p_{s_n,p}+ \| u^n \|_{L^p(\Omega)}^p\right)<\infty,
\end{equation} 
then
\begin{equation}
\label{bernardTesi}
\lim_{n\to\infty}
(1-s_n)\big| [\tilde u^n]^p_{s_n,p,f_n}-[\tilde u^n]^p_{s_n,p,f} \big|=0.
\end{equation}
\end{lemma}

\begin{proof}
Since we can estimate
\begin{align*}
      \big| [\tilde u^n]^p_{s_n,p,f_n}-[\tilde u^n]^p_{s_n,p,f} \big| 
      &
      \leq
      [\tilde u^n]^p_{s_n, p} \sup_{(x,y) \in \R^{2d}} \vert f_n(x)f_n(y)-f(x)f(y)\vert
      \\
      &
      \leq
      2 [\tilde u^n]^p_{s_n, p}  \sup\{ \| f_n \|_{\infty} : n \in \N\}\| f_n -f \|_{\infty},
\end{align*}
the conclusion immediately follows from~\eqref{eq:bound_lemmetto_seq_pesi}.
\end{proof}

\begin{lemma}
\label{Lem208012025}
If $(u^n)_{n\in\N}\subset L^p(\Omega)$ is such that
\begin{equation}
\label{quprovo1}
\sup_{n \in \N}\left( (1-s_n)[\tilde u^n]^p_{s_n,p,f_n}+ \|u^n \|_{L^p(\Omega)}^p\right)<\infty, 
\end{equation}
then
\begin{equation*}
\sup_{n \in \N} (1-s_n)[\tilde u^n]^p_{s_n,p}+ \| u^n \|_{L^p(\Omega)}^p
< 
\infty.
\end{equation*}
\end{lemma}

\begin{proof}
On the one hand, by the inequality $(a+b)^p\le 2^p(a^p+b^p)$ for $a,b\ge0$, exploiting symmetry, and passing to the $d$-dimensional spherical coordinates, we can estimate
\begin{equation}
\begin{split}
\int_{\R^d} \int_{B_1(y)^c}
&
\frac{|\tilde{u}^n(x)-\tilde{u}^n(y)|^p}{\|x-y\|^{d+s_n p}}  \,\mathrm{d}x \,\mathrm{d}y 
\le
2^p \int_{\R^d} \int_{B_1(y)^c} \frac{|\tilde u^n (x)|^p+|\tilde{u}^n(y)|^p}{\|x-y\|^{d+s_n p}} \,\mathrm{d}x \,\mathrm{d}y
\\
& 
\leq 
2^{p+1}\int_{\R^d} |\tilde u^n(y)|^p \int_{B_1(y)^c} \frac{\mathrm{d}x}{\|x-y\|^{d+s_n p}} \,\mathrm{d}y 
\\
&
=
2^{p+1} \| u^n \|^p_{p} \int_{1}^{\infty} \frac{d\omega_d}{\rho^{1+s_n p}} \,\mathrm{d} \rho 
=
\frac{C(d, p) \|u^n\|^p_p}{s_n}
\le
C(d, p)  \|u^n\|^p_p
\end{split}
\label{bern1234}
\end{equation}
for all $n\in\N$, since $s_n\to 1^-$, thus we can assume $s_n > \frac{1}{2}$.
On the other hand, recalling the notation in~\eqref{ingrass}, since $\tilde u^n(y)$ is supported on~$\Omega$, then $\tilde{u}^n(x)$ for $x\in B_1(y)$ is supported on $\Omega_1$, and since $f_n\to f$ uniformly and $f>0$, we have
\begin{equation}
\begin{split}
        \int_{\R^d} \int_{B_1(y)} 
        &\frac{|\tilde{u}^n(x)-\tilde{u}^n(y)|^p}{\|x-y\|^{d+s_n p}}  \,\mathrm{d}x \,\mathrm{d}y
        =
        \int_{\Omega_1} \int_{\Omega_1} \frac{|\tilde{u}^n(x)-\tilde{u}^n(y)|^p}{\|x-y\|^{d+s_n p}}  \,\mathrm{d}x \,\mathrm{d}y
        \\
        & 
        \leq
        \int_{\Omega_1} \int_{\Omega_1} \frac{|\tilde{u}^n(x)-\tilde{u}^n(y)|^p}{\|x-y\|^{d+s_n p}}  \frac{f_n(x)f_n(y)}{ \inf_{ (x,n) \in \Omega_1 \times \N } f_n^2(x)} \,\mathrm{d}x \,\mathrm{d}y 
        \\
        & 
        = 
        \frac{[\tilde u^n]^p_{s_n,p,f_n}}{\inf_{ (x,n) \in \Omega_1 \times \N } f_n^2(x)}
=
C(\Omega_1, f_n) [\tilde u^n]^p_{s_n,p,f_n}.
\end{split}
\label{bern123}
\end{equation}
Since $f_n\to f$ uniformly and $f>0$, the constant $C$ can be made independent of $n$, and only dependent on the uniform limit $f$. The conclusion hence follows by combining~\eqref{bern1234} and~\eqref{bern123}, multiplying by $(1-s_n)$, and owing to the assumption~\eqref{quprovo1}.
We omit the plain details.
\end{proof}

\section{Proof of \texorpdfstring{\cref{11genstab}}{Theorem 1.2}}\label{sec:flow}

\subsection{Stability of Hilbertian gradient flows}

We briefly recall some abstract machinery from~\cite{CdLKNP23} concerning Hilbertian gradient flows.

Let $\h$ be a Hilbert space endowed with scalar product $\langle\cdot,\cdot\rangle_\h$ and  norm $\|\cdot\|_\h$.
Given $\F \colon \h \to (-\infty,+\infty]$, we let $\D(\F) = \{\, x\in\h \colon \F(x)< + \infty\, \}$ and
\begin{equation*}
\partial\F(x)
= 
\left\{\,v\in\h \colon \liminf_{y\to x} \frac{\F(y)-\F(x)-\langle v,y-x\rangle_\h}{\|y-x\|_\h }\ge0 \,\right\}
\end{equation*}
be the \emph{subdifferential} of $\F$ at $x\in\D(\F)$.

We recall the following result, which is a particular case of~\cite{CdLKNP23}*{Prop.~3.7}. 
Here and below, given any vector space $\mathscr{V}$, we let $\mathscr{V}^*$ be the algebraic dual space of $\mathscr{V}$ and $\mathscr{V}'$ the topological dual space of $\mathscr{V}$. 

\begin{proposition}
\label{singleton}
Let $\K$ be a dense subspace of $\h$.
Let $\F\colon\h\to(-\infty,+\infty]$ be a proper, convex, and strongly lower semicontinuous functional and let $x\in\D(\F)$.  
If there exists $T\in\K^*$ such that
\[
\lim_{t\to 0}\frac{\F(x+ty)-\F(x)}{t}=T(y),\quad\textrm{for every }y\in\K\,,
\]
then either $\partial\F(x)=\emptyset$ or $\partial\F(x)=\{v\}$, where $v$ is the (unique) element in $\h$  satisfying $T(y)=\langle v,y\rangle_\h$ for every $y\in\K$. 
In particular, $T\in \K'$ and $v$ is its  unique continuous extension to $\h'$.
\end{proposition}

We also recall the following stability result, which is contained in~\cite{CdLKNP23}*{Thm.~3.8}.

\begin{theorem}
\label{genstab}
Let $\F_n\colon\h\to(-\infty,+\infty]$ be a proper, strongly lower semicontinuous, convex, and positive functional for every $n\in\N$.
Assume the following:
\begin{enumerate}[label=(\alph*),itemsep=1ex]

\item 
$(\F_n)_{n\in\N}$ $\Gamma$-converges to some proper functional $\F_\infty\colon\h\to(-\infty,+\infty]$ with respect to the strong $\h$-convergence;

\item
any bounded sequence 
$(x_n)_{n\in\N}\subset\h$ such that 
$\displaystyle\sup_{n\in\N}\F_n(x_n)<\infty$ admits a strongly $\h$-convergent subsequence.

\end{enumerate}

If $(x^n_{0})_{n\in\N}\subset\h$ is such that $x^n_{0}\in \D(\F_n)$ for every $n\in\N$, $\displaystyle\sup_{n\in\N}\F_n(x^n_{0})<\infty$ and 
$x^n_{0}\to x^\infty_{0}$ strongly in $\h$ for some  $x^\infty_{0}\in\h$, then the following hold:
\begin{enumerate}[label=(\roman*),itemsep=1ex]

\item 
$x^\infty_{0}\in\D(\F_\infty)$;

\item
for every $T>0$, the problem 
\begin{equation*}
\begin{cases}
\dot{x}(t)\in -\partial \F_n(x(t)),  \quad \text{for a.e.\ }t \in (0,T)\,, \\[1ex]
x(0)=x^n_{0},
\end{cases}
\end{equation*}
admits a unique solution $x_n\in H^1([0,T];\h)$  for every $n\in\N \cup \{\infty\}$;

\item
$(x_n)_{n\in\N}$ weakly converges to $x_\infty$ in $H^1([0,T];\h)$.

\end{enumerate}
Moreover, 
if $\displaystyle\lim_{n\to\infty} \F^n(x^n_{0})=\F_\infty(x^\infty_{0})$, then actually $(x_n)_{n\in\N}$ strongly converges to $x_\infty$ in $H^1([0,T];\h)$ and also
\begin{equation*}
x_n(t)\xrightarrow{\h}x_{\infty}(t)
\quad
\text{and}
\quad
\F_n(x_n(t))\to\F_{\infty}(x_{\infty}(t))
\quad\text{for every}\ t\in[0,T].
\end{equation*}
\end{theorem}

\subsection{Proof of \texorpdfstring{\cref{11genstab}}{Theorem 1.2}}

The validity of \cref{11genstab} follows by combining the abstract results above with the following proposition.
Here and below, $\langle\cdot,\cdot\rangle$ denotes the standard scalar product in $L^2(\R^d)$.

\begin{proposition}
\label{lemma:firstvar}
Let $\varphi\in C_c^\infty(\Omega)$ and $u\in L^2(\Omega)$.
The following hold:
\begin{enumerate}[label=(\roman*),itemsep=1ex]

\item 
if $u\in H^1_0(\Omega)$, then
\begin{equation}
\label{eq:firstvar}
\lim_{t\to 0}\frac{\|\nabla(u+t\varphi)\|^2_{2,f^2}-\|\nabla u\|_{2,f^2}^2}{t}=\langle(-\mathfrak{D})^{f} u,\varphi \rangle;
\end{equation}

\item 
if $ [\tilde{u}]_{s,2,f} < \infty$ for some $s\in (0,1)$, then
\begin{equation}
\label{eq:firstvar_fract}
\lim_{t\to 0}\frac{[u+t\varphi]_{s,2,f}^2-[u]_{s,2,f}^2}{t}=\langle(-\mathfrak{D})^{s,f} u,\varphi \rangle.
\end{equation}

\end{enumerate}
\end{proposition}

\begin{proof}
We only prove~\eqref{eq:firstvar_fract}, the proof of~\eqref{eq:firstvar} being straightforward. 
We note that
\begin{align*}
[u+t\varphi]_{s,2,f}^2
&
= 
[u]_{s,2,f}^2
+
t^2 [\varphi]_{s,2,f}^2
\\
&
\quad
+2 t \int_{\R^d} \int_{\R^d} \frac{(u(x)-u(y))(\varphi(x)-\varphi(y))}{\|x-y\|^{d+2s}}f(x)f(y) \,\mathrm{d}x\,\mathrm{d}y
\end{align*}
for every $t\in\R$, and thus we easily get that 
\begin{align*} 
      \lim_{t\to 0} \frac{[u+t\varphi]_{s,2,f}^2-[u]_{s,2,f}^2}{t} 
      &
      = 
      2 \int_{\R^d} \int_{\R^d} \frac{(u(x)-u(y))(\varphi(x)-\varphi(y))}{\|x-y\|^{d+2s}}f(x)f(y) \,\mathrm{d}x \,\mathrm{d}y
      \\
      &
      =
      4 \int_{\R^d} u(x)f(x) \lim_{r \to  0^+} \int_{\R^d \setminus B_r} \frac{\varphi(x)-\varphi(y)}{\|x-y\|^{d+2s}} f(y) \,\mathrm{d}y \,\mathrm{d}x
      \\
      &
      =
      \langle u,(- \mathfrak{D})^{s,f} \varphi \rangle= \langle (- \mathfrak{D})^{s,f} u,\varphi \rangle
\end{align*}
in virtue of the distributional definition in~\eqref{def:Dsf}, 
concluding the proof.
\end{proof}

We are now ready to prove \cref{11genstab}.

\begin{proof}[Proof of \cref{11genstab}]
Since $f_n\to f$ uniformly and $f>0$, the functionals given by $\F_n(u)=(1-s_n)[\tilde{u}]_{s_n,2,f_n}^2$ are positive for $n\gg1$. Further, they are easily shown to be convex. 
By \cref{thmGammaconvs1stab}\ref{item:G-liminf}--\ref{item:G-limsup}, $(\F_n)_{n\in\N}$ $\Gamma$-converges to the functional $\F(u)= K_{d,2}\| \nabla u \|^2_{2,f^2}$ in the strong topology of $L^2(\Omega)$, whereas by \cref{thmGammaconvs1stab}\ref{item:compactness} every sequence $(u_n)_{n\in\N}\subset L^2(\Omega)$ such that $\sup_{n\in\N}\F_n(u_n)<\infty$ admits a strong $L^2$-limit $u\in L^2(\Omega)$. 
The conclusion hence follows by \cref{genstab,singleton} combined with \cref{lemma:firstvar}.
\end{proof}

%%% BIBLIO %%%%

\end{document}